\documentclass[11pt,a4paper,reqno]{amsart}
\usepackage{amsmath,amssymb,amsfonts,epsfig,mathrsfs,cite}
\usepackage[T1]{fontenc}
\usepackage{color}
\usepackage{array}
\usepackage{amsthm}
\usepackage{amstext}
\usepackage{graphicx}
\usepackage{setspace}

\usepackage[title]{appendix}

\makeatletter
\@namedef{subjclassname@2020}{%
  \textup{2020} Mathematics Subject Classification}
\makeatother

\usepackage[margin=2.5cm]{geometry}
\usepackage{color}
\usepackage{enumitem}
\usepackage{amscd,psfrag}
\usepackage{yhmath}
\usepackage[mathscr]{eucal}

\usepackage{comment}

\allowdisplaybreaks[4]

\usepackage{slashed}

\makeatletter
\pdfpageheight\paperheight
\pdfpagewidth\paperwidth

\usepackage{epstopdf}
\usepackage{chngcntr}
\counterwithin{figure}{section}
\usepackage{mathrsfs}

\usepackage{indentfirst}	

\usepackage[normalem]{ulem}
\theoremstyle{plain}
\numberwithin{equation}{section}

\newtheorem{definition}{Definition}[section]
\newtheorem{theorem}[definition]{Theorem}
\newtheorem*{theorem*}{Main Theorem}\newtheorem*{th*}{Theorem}\newtheorem*{definition*}{Definition}\newtheorem*{example*}{Example}
\newtheorem*{proposition*}{Proposition}

\newtheorem{remark}[definition]{Remark}

\newtheorem*{remark*}{Remark}
\newtheorem*{sideremark*}{Side Remark}

\newtheorem*{claim*}{Claim}
\newtheorem*{q*}{Question}
\newtheorem{lemma}[definition]{Lemma}
\newtheorem{corollary}[definition]{Corollary}
\newtheorem*{corollary*}{Corollary}\newtheorem*{cq*}{Conjecture/Question}

\newtheorem{proposition}[definition]{Proposition}

\newcommand{\R}{\mathbb{R}}

\newcommand{\emb}{\hookrightarrow}

\newcommand{\e}{\varepsilon}

\newcommand{\dd}{{\rm d}}

\newcommand{\M}{{\mathcal{M}}}

\newcommand{\CC}{{\mathscr{C}}}
\newcommand{\sph}{{\mathbf{S}}}

\newcommand{\length}{{\rm Length}}

\newcommand{\GG}{{\mathcal{G}}}

\def\XXint#1#2#3{{\setbox0=\hbox{$#1{#2#3}{\int}$ }
\vcenter{\hbox{$#2#3$ }}\kern-.6\wd0}}

\title{Gaussian kernels on non-simply-connected closed Riemannian manifolds are never positive definite}

\author{Siran Li}

\address{Siran Li: School of Mathematical Sciences $\&$ CMA-Shanghai, Shanghai Jiao Tong University, No.~6 Science Buildings,
800 Dongchuan Road, Minhang District, Shanghai, China (200240)}

\email{\texttt{siran.li@sjtu.edu.cn}}

\keywords{Gaussian kernel; kernel method; isometric embedding; closed geodesic; RKHS (reproducing kernel Hilbert space); non-simply-connected manifold.}

\subjclass[2020]{46E22, 46N30, 68R12, 62B11}
\date{\today}

\begin{document}

\begin{abstract}
We show that the Gaussian kernel $\exp\left\{-\lambda d_g^2(\bullet, \bullet)\right\}$ on any non-simply-connected closed Riemannian manifold $(\M,g)$, where $d_g$ is the geodesic distance, is not positive definite for any $\lambda > 0$, combining  analyses in the recent preprint~\cite{cmo} by Da Costa--Mostajeran--Ortega and classical comparison theorems in  Riemannian geometry.  
\end{abstract}
\maketitle

\section{Main Result}\label{sec: main result}
This contribution is concerned with mathematical properties of the standard Gaussian kernel defined via the  geodesic distance on Riemannian manifolds. Our main result is as follows:

\begin{theorem*}
Let $\M$ be any compact, non-simply-connected $\CC^\infty$-manifold without boundary. For any Riemannian metric $g$ on $\M$ and any parameter $\lambda >0$, the Gaussian kernel $e^{-\lambda d_g^2(\bullet,\bullet)}: \M \times \M \to [0,\infty[$ cannot be  positive definite. Here $d_g$ is the geodesic distance on $(\M,g)$. 
\end{theorem*}

The proof relies crucially on the recent nice work~\cite{cmo} by Da Costa--Mostajeran--Ortega, which proves the above theorem for $\M$ = the round circle $\varrho\sph^1$ for any $\varrho >0$, and hence for spheres, real projective spaces, and real Grassmannians, etc. The arguments in \cite{cmo} combined with classical calculus of variations techniques on manifolds will enable us to conclude the theorem.

\section{Background and preliminaries}\label{sec: intro}

\subsection{Manifold learning meets kernel methods}
In recent developments of machine learning, the field of manifold learning, which focuses on non-Euclidean geometric features of  data, has attracted considerable attention. See \cite{bns, cal, ss} among many other references. We are primarily interested in its connexion with the kernel methods, which also play a prominent role in machine learning and its  mathematical theories. In \cite{cs} Cucker--Smale surveyed in detail the mathematical  foundations of kernel methods. Applications to various problems including support vector machines, kernel  principal component analysis, controlled stochastic differential equations, and reservoir computing, etc., can be found in \cite{scfly, ss, five, ssm, go} and the many references cited therein.

Consider now a more general setting: $(X,d)$ = metric space. For kernel methods applied to metric spaces, the essential mathematical structure is a symmetric function $k: X \times X \to \R$, known as a \emph{kernel}, which depends on the metric distance between observations. A widely used family of kernels on $(X,d)$ is 
\begin{align*}
k(x,y) := e^{-\lambda \left[d(x,y)\right]^q}\qquad \text{for } \lambda, q >0.
\end{align*}
See Chapelle--Haffner--Vapnik \cite{chv} and Feragen--Lauze--Hauberg
\cite{flh}. 

When $q=2$, such $k$ is known as the \emph{Gaussian kernel} on $(X,d)$.

Any complete Riemannian manifold $(\M,g)$ has a metric space structure,  with metric space topology compatible with the manifold topology.  The metric ---  \emph{i.e.}, the distance function --- on $\M$ with respect to the Riemannian metric $g$ is given by $$d_g(x,y):= \inf \left\{\gamma: [0,L] \to \M:\,\text{$\gamma$ is a piecewise $\CC^1$-curve with $\gamma(0)=x$ and $\gamma(1)=y$}\right\}.$$ It makes no difference if one replaces here ``piecewise $\CC^1$''  by ``piecewise $\CC^\infty$'', ``$\CC^\infty$'', or ``$\CC^1$''. Completeness of the metric space $(\M,d_g)$ is equivalent to geodesic completeness  of the Riemannian manifold $(\M,g)$. See, \emph{e.g.},  \cite[Hopf--Rinow Theorem, \S 7]{dc}.

\subsection{Positive definite kernels}
For applications in machine learning (\cite{scfly, ss, ssm, go, flh, cs, five}), one usually requires that the kernel function maps data into an RKHS (reproducing kernel Hilbert space), which in turn entails \emph{positive definiteness} of the kernel. See Sch\"{o}lkopf--Smola \cite{ss}.
\begin{definition}\label{def: pos def}
$k: (X,d) \times (X,d) \to \R$ is positive definite  if for any $N \in \mathbb{N}$ and $\{x_1, \ldots, x_N\} \subset X$, the Gram  matrix $\GG= \left\{\GG_{ij}\right\} \in {\rm Mat}(N \times N;\R)$ with $\GG_{ij} := k(x_i,x_j)$ is positive semi-definite. 
\end{definition}

This notion of positive definiteness of kernel functions was proposed by Mathias and Bochner.  See \cite{boc, s}.

Positive definite kernels are also of considerable theoretical interest in functional analysis. A classical theorem of Schoenberg (\cite{s}) back in 1938 states that a separable metric space $(X,d)$ is isometrically embeddable into the Hilbert space $L^2(]0,1[)$ if and only if the Gaussian kernel $e^{-\lambda d^2(\bullet,\bullet)}$ is positive definite for each $\lambda>0$. 
One thus finds that $L^2(]0,1[)$ contains  copies of $\R^n$, the space ${\bf Sym}^{++}_n$ of symmetric positive  definite $n \times n$ matrices with the Frobenius norm/log-Euclidean metric, and the real Grassmannians ${\bf Gr}(k,n)$ with the projective metric, etc. 

\begin{definition}\label{def: text}
Let $(X,d)$ and $(X',d')$ be metric spaces. A map $\iota: X \to X'$ is said to be an isometric embedding if $\iota$ is one-to-one and $d'\big(\iota(x), \iota(y)\big)=d(x,y)$ for all $x, y \in X$.

For complete Riemannian manifolds $(\M,g)$ and $(\M',g')$, a map $\iota: \M \to \M'$ is said to be a metric space isometric embedding if $\iota$ is an embedding in the category of $\CC^\infty$-manifolds, and $\iota: \left(\M,d_g\right) \to \left(\M,d_{g'}\right)$ is an isometric embedding of metric spaces. 
\end{definition}

\begin{remark}\label{remark}
For Riemannian manifolds, an isometric embedding in the sense of metric spaces is very different from that in the sense of differential geometry. See Appendix~\ref{sec: appendix}. 
\end{remark}

On the other hand, various results on the negative side are reported in the literature. Denote the range of positive definiteness of the Gaussian kernel on $(X,d)$ as \begin{align*}
\Lambda_+(X):=\left\{\lambda \in ]0,\infty[:\, \text{the Gaussian kernel } e^{-\lambda d^2(\bullet,\bullet)} \text{ is positive definite}\right\}.
\end{align*}
Feragen--Lauze--Hauberg \cite{flh} proved that for any Riemannian manifold $(\M,g)$, the range $\Lambda_+(\M,d_g)$ cannot be the whole $]0,\infty[$ unless $g$ is Euclidean\footnote{Indeed, \cite[Theorem~1]{flh} proves a more general result: the only geodesic metric spaces with positive definite Gaussian kernel are in the class ${\bf CAT}(0)$.}. Thus, by Schoenberg's theorem \cite{s}, only flat Riemannian manifolds can be isometrically embedded \emph{in the metric space sense} into $L^2(]0,1[)$. 

This, however, leaves open the important question:
\emph{which Riemannian manifolds have a large $\Lambda_+ \subset ]0,\infty[$}? By explicit computations for the Gaussian kernel associated to equidistributed points on $\sph^1$, in the recent preprint~\cite{cmo}, Da Costa--Mostajeran--Ortega proved the following
\begin{theorem}\label{thm: Sn etc}
The range of positive definiteness of the Gaussian kernel on the following Riemannian manifolds are empty (namely that $\Lambda_+= \emptyset$), all equipped with canonical metrics:
\begin{itemize}
\item
$\sph^n$ with $n \geq 1$;
\item
the real project spaces $\R\mathbf{P}^n$ with $n>1$;
\item
the real Grassmannians ${\bf Gr}(k,n)$ with $1 \leq k <n$;
\item
the hyperbolic hyperboloid; and 
\item 
the torus $\mathbf{T}(a,b) = \left\{\begin{bmatrix}
(a+b\cos \theta)\cos \phi\\
(a+b\cos \theta)\sin \phi\\
b\sin\theta
\end{bmatrix}\in\R^3:\,0 \leq \theta,\phi < 2\pi\right\}$ for any $0<a<b$.
\end{itemize}
\end{theorem}

In effect, our main theorem in \S\ref{sec: main result} adds a vast class to the list of Riemannian manifolds  whose Gaussian kernel is not positive definite for any $\lambda>0$. A manifold is said to be \emph{closed} if it is compact and boundaryless. Combining our result with Theorem~\ref{thm: Sn etc} above, we deduce 
\begin{corollary}\label{cor}
$\Lambda_+(\M,g)=\emptyset$ for the following {\bf closed} Riemannian manifolds:
\begin{itemize}
\item
Any non-simply-connected $(\M,g)$;
\item
$\sph^n$ with $n \geq 2$;
\item
the real Grassmannians ${\bf Gr}(k,n)$ with $1 \leq k <n$.
\end{itemize} 
\end{corollary}

The Gaussian kernel represents a critical or borderline case, in view of the following result of Istas \cite[Theorem~2.12]{istas}. See also \cite[Theorem~3]{flh}.
\begin{proposition}
On any complete Riemannian manifold $(\M,g)$ and for any $q>2$, there is some $\lambda>0$ such that the kernel $k(x,y) := e^{-\lambda \left[d_g(x,y)\right]^q}$ fails to be positive definite.
\end{proposition}

\subsection{Further comments}

It is remarked in \cite{cmo} that $\Lambda_+\left(\sph^1\right)=\emptyset$ --- hence Theorem~\ref{thm: Sn etc} --- can be deduced via alternative arguments from Wood~\cite{w} and Gneiting \cite{gn}. Nevertheless, the hands-on Fourier theoretic computations in \cite{cmo} are particularly suitable for extensions to more general classes of Riemannian manifolds, as will be shown in this contribution.

On the other hand, \cite{cmo} suggests that one may deduce $\Lambda_+(\M,g)=\emptyset$ for any closed Riemannian manifold, by showing that any such $(\M,g)$ admits an isometric embedding of the round circle $\sph^1$ in the metric space sense. This, in fact, seems to impose rather strong limitations on $(\M,g)$. See the theorem in Appendix~\ref{sec: appendix}. But at the moment we are unable to provide a counterexample, nor to confirm the suggestion in \cite{cmo}.

Let us also remark that variants of Gaussian kernels on $\sph^n$ and more general Riemannian manifolds are studied in theoretical works in mathematics, statistics, and machine learning. See \cite{btmd, fh, flh, gn} and many others.  

\section{Non-positive definiteness of Gaussian kernel when $\pi_1(\M) \neq \{0\}$}

In this section we prove our main theorem, which is stated in \S\ref{sec: main result} and reproduced below:
\begin{theorem}\label{thm: main}
Let $\M$ be an arbitrary closed (\emph{i.e.}, compact and boundaryless) differentiable manifold such that $\pi_1(\M) \neq \{0\}$. For any Riemannian metric $g$ on $\M$ and any parameter $\lambda >0$, the Gaussian kernel $e^{-\lambda d_g^2(\bullet,\bullet)}: \M \times \M \to [0,\infty[$ cannot be  positive definite. 
\end{theorem}

That is, we prove for non-simply-connected closed manifolds $(\M,g)$ that
\begin{align*}
\Lambda_+(\M,g) := \left\{\lambda \in [0,\infty[:\, \text{$e^{-\lambda d_g^2(\bullet,\bullet)}: \M \times \M \to [0,\infty[$ is positive definite}
\right\} = \emptyset.
\end{align*}

The key to the proof is the following linear algebraic lemma, established by Da Costa--Mostajeran--Ortega \cite[pp.5--7]{cmo} by explicit computations:

\begin{lemma}[See \cite{cmo}]\label{lemma: cmo}
Given any $\lambda>0$. 
The circulant matrix 
\begin{align*}
\mathcal{K} = \begin{bmatrix}
1 & \exp\left(-\lambda \left(\frac{2\pi}{N}\right)^2\right) & \exp\left(-\lambda \left(\frac{4\pi}{N}\right)^2\right) & \cdots & \exp\left(-\lambda \left(\frac{2\pi}{N}\right)^2\right)\\
\exp\left(-\lambda \left(\frac{2\pi}{N}\right)^2\right) &1&\cdots &\cdots &\exp\left(-\lambda \left(\frac{4\pi}{N}\right)^2\right)\\
\vdots&\vdots&\ddots&\ddots&\vdots\\
\vdots&\vdots&\ddots&\ddots&\vdots\\
\exp\left(-\lambda \left(\frac{2\pi}{N}\right)^2\right)&\exp\left(-\lambda \left(\frac{4\pi}{N}\right)^2\right)&\exp\left(-\lambda \left(\frac{6\pi}{N}\right)^2\right)&\cdots&1
\end{bmatrix}
\end{align*}
has a negative eigenvalue for some sufficiently large natural number $N$ divisible by 4.
\end{lemma}

Lemma~\ref{lemma: cmo} is used in \cite{cmo} to prove that the Gaussian kernel is never positive definite on any manifold admitting a metric space isometric embedding of the round circle $\sph^1$. Our key observation is that Lemma~\ref{lemma: cmo}  combined with classical comparison theorems in Riemannian geometry (see \emph{e.g.} \cite{dc, ce}) enables us to deduce our main Theorem~\ref{thm: main}.

Throughout, by a round circle we mean  $\sph^1$ equipped with the canonical measure, \emph{i.e.}, the measure induced by the 2-dimensional Lebesgue measure via the inclusion of $\sph^1 \subset \R^2$ as the unit circle $\{z \in \R^2:\,|z|=1\}$.

\begin{proof}[Proof of Theorem~\ref{thm: main}] 

Let $\M$ be a closed, non-simply-connected $\CC^\infty$-manifold. Fix a Riemannian metric $g$ on $\M$ and a positive number $\lambda$. Our arguments rely essentially on the classical theorem \emph{\`{a} la} \'{E}. Cartan (\cite{cartan}; see also do Carmo \cite[Theorem~12.2.2]{dc}):
\begin{th*}
For any Riemannian manifold $(\M,g)$ with $\pi_1(\M) \neq \{0\}$, in each nontrivial $1$-homotopy class there is a  \underline{shortest} closed geodesic. 
\end{th*}

We want to pick ``the shortest of all such shortest closed geodesics''. Note, however, that there may be infinitely many such geodesics, as $\pi_1(\M)$ is only known to be finitely generated (but not finite) in general. To avoid complications due to  non-uniqueness of generators, we resort to the following arguments, which allow us to select the ``almost'' shortest nontrivial closed geodesic on $(\M,g)$.

For any (piecewise differentiable) loop $\gamma$, we write 
\begin{align*}
\Omega_{[\gamma]}:= \text{ the free homotopy class of loops containing $\gamma$.} 
\end{align*}
Then set
\begin{align*}
\kappa_0 := \inf \left\{ \length[\gamma]:\, \text{$\gamma$ is a shortest closed geodesic in $\Omega_{[\gamma]}$, with } \Omega_{[\gamma]} \neq \{0\} \text{ in $\pi_1(\M)$} \right\}.
\end{align*}
The infimum is taken over all closed geodesics $\gamma$ and then over all free homotopy classes $\Omega_{[\gamma]}$.

First, we claim that
 \begin{align*}
 \kappa_0 > 0.
\end{align*}  
To see this, if $\kappa_0=0$ then there  would be a closed geodesic $\gamma'$,  which is arbitrarily short, with $\Omega_{[\gamma']} \neq \{0\}$. But as $\M$ is a compact Riemannian manifold, there is a positive number $\eta_0 = \eta_0(\M)$ such that any closed curve with length less than $\eta_0$ must be null homotopic. Indeed, one may pick a finite open covering on $\M$ consisting of geodesic normal charts, and take $\eta_0$ to be the Lebesgue number of this covering. This leads to contradiction.

For an $\e \in \left]0, \frac{\kappa_0}{100}\right]$ to be specified later, we pick $\gamma_\star$, a shortest closed geodesic in its own free homotopy class (which is non-trivial), such that 
\begin{align*}
\length\left[\gamma_\star\right] \leq \kappa_0 + \e.
\end{align*}
Then, for some big number $N \in \mathbb{N}$ to be specified later, we pick $N$ points $\{x_1, \ldots, x_N\} \subset \gamma_\star$ equidistributed with respect to the  metric-induced distance $d_g$ \emph{restricted to $\gamma_\star$}. In formula, for any $1 \leq j < i \leq N$ we require that 
\begin{equation*}
\left[d_g\big|_{\gamma_\star}\right]\left(x_i,x_j\right) := \begin{cases}
\frac{2\pi}{N} (i-j)\qquad \text{ if } i-j \leq \frac{N}{2},\\
\frac{2\pi}{N} (j+N-i)\qquad \text{ if } i-j > \frac{N}{2}.
\end{cases}
\end{equation*}
In other words, $\{x_1, \ldots, x_N\}$ on $\gamma_\star$ satisfy the same distance relations --- when considered with respect to $d_g$ on $\gamma_\star$ --- as for $N$ equidistributed points on the round circle $\sph^1$. This can be done as $\gamma_\star$ is diffeomorphic to $\sph^1$ by construction; in particular, the shortest condition implies the simplicity or primality of $\gamma_\star$.

To proceed, we claim that the distance relations of these $N$ points are almost the same when considered with respect to $d_g$ on the whole manifold $\M$. More precisely:
\begin{equation*}
0 \leq \left[d_g\big|_{\gamma_\star}\right]\left(x_i,x_j\right) -\left[d_g\big|_{\M}\right]\left(x_i,x_j\right) \leq \e\qquad \text{for any }1 \leq j < i \leq N.\tag{$\spadesuit$}
\end{equation*}

\begin{proof}[Proof of Claim~($\spadesuit$)]
The non-negativity of $\left[d_g\big|_{\gamma_\star}\right]\left(x_i,x_j\right) -\left[d_g\big|_{\M}\right]\left(x_i,x_j\right)$ is clear, as  distances are defined by taking infima.

To prove the other inequality, suppose for some $i,j$ that the difference where greater than $\e$. Then there would be a curve $\sigma$ ({\bf s}igma for ``{\bf s}hort-cut'') with endpoints $x_i$, $x_j$ whose concatenation with $\gamma_\star^+$ is a piecewise-$\CC^\infty$ loop satisfying
\begin{align*}
\length\left[\sigma \,\#\, \gamma^+_\star\right] <  \length\left[\gamma_\star\right] -\e.
\end{align*}
Here and hereafter, $\gamma_\star^+$ denotes the longer of the two geodesic segments on $\gamma_\star$ with endpoints $x_i$ and $x_j$, and $\gamma_\star^-$ denotes the shorter one. (The choice is clear when the two segments have equal length, for which case we may safely omit the details.)

The above inequality  implies  
\begin{align*}
\length\left[\sigma \,\#\, \gamma^+_\star\right] <  \kappa_0.
\end{align*}
Hence, the loop $\sigma \,\#\, \gamma^+_\star$ must be null-homotopic, in view of the minimality condition in  definition of $\kappa_0$. In this case, however, the loop $\sigma \,\#\,\gamma^-_\star$ cannot be null-homotopic, for otherwise $\gamma_\star$ must be null-homotopic, which can be seen by contracting $\sigma \,\#\, \gamma^+_\star$ and $\sigma \,\#\, \gamma^-_\star$ sequentially.

 Therefore, from the shortness of $\sigma$ (\emph{i.e.}, $\length[\sigma] < \length\left[\gamma_\star^-\right]-\e$), the definition of $\gamma_\star^\pm$, and the choice for $\gamma_\star$, we deduce that
\begin{align*}
\length\left[\sigma \,\#\,\gamma^-_\star\right] &< 2\,\length\left[\gamma^-_\star\right] - \e\\
&\leq \length\left[\gamma_\star^+\right] + \length\left[\gamma_\star^-\right]- \e\\
&= \length\left[\gamma_\star\right] - \e\\
&\leq \kappa_0.
\end{align*}
 This contradicts the definition of $\kappa_0$.   \end{proof}
 
 Now let us consider the Gram matrix $$\GG=\GG_N \in {\rm Mat}\left(N \times N;\R_+\right)$$ for the Gaussian kernel with parameter $\lambda>0$. Its entries are
 \begin{align*}
 \GG_{ij} =e^{-\lambda d_g^2(x_i, x_j)}\qquad \text{ for } i,j \in \{1,2,\ldots,N\},
 \end{align*}
 where $\{x_1, \ldots, x_N\}$ are the points on the closed geodesic $\gamma_\star$ selected above. Here, the distance function $d_g$ must be understood as $\left[d_g\big|_{\M}\right]$ rather than the restricted version.

  The estimate ($\spadesuit$) and Taylor expansion applied to $e^{-\lambda (A+\e)} - e^{-\lambda A}$ for $A>0$ imply that
 \begin{align*}
 \left\|\GG-\mathcal{K}\right\|_\infty := \max_{i,j \in \{1,2,\ldots,N\}}  \left|\GG_{ij}-\mathcal{K}_{ij}\right| \leq C_0\e,
 \end{align*}
where $C_0=C_0(\lambda)$ and $\mathcal{K}$ is  the circulant matrix appearing in Lemma~\ref{lemma: cmo}: 
\begin{align*}
\mathcal{K} = \begin{bmatrix}
1 & \exp\left(-\lambda \left(\frac{2\pi}{N}\right)^2\right) & \exp\left(-\lambda \left(\frac{4\pi}{N}\right)^2\right) & \cdots & \exp\left(-\lambda \left(\frac{2\pi}{N}\right)^2\right)\\
\exp\left(-\lambda \left(\frac{2\pi}{N}\right)^2\right) &1&\cdots &\cdots &\exp\left(-\lambda \left(\frac{4\pi}{N}\right)^2\right)\\
\vdots&\vdots&\ddots&\ddots&\vdots\\
\vdots&\vdots&\ddots&\ddots&\vdots\\
\exp\left(-\lambda \left(\frac{2\pi}{N}\right)^2\right)&\exp\left(-\lambda \left(\frac{4\pi}{N}\right)^2\right)&\exp\left(-\lambda \left(\frac{6\pi}{N}\right)^2\right)&\cdots&1
\end{bmatrix}.
\end{align*}

To conclude, recall from Lemma~\ref{lemma: cmo} (\emph{cf.} Da Costa--Mostajeran--Ortega \cite[pp.5--7]{cmo}) that $\mathcal{K}$ has a negative eigenvalue for sufficiently large $N=N(\lambda) \equiv 0 \mod 4$. On the other hand, having a negative eigenvalue is an open condition, and eigenvalues vary continuously with respect to the matrix entries\footnote{This follows from the classical perturbation theorem of Kato \cite[pp.126--127, Theorem 5.2]{kato}. It was claimed in \cite{g}, where the famous Ger\u{s}gorin's disc theorem was established, but the arguments appear to be incomplete. See the recent work~\cite{lz} for a rigorous proof and discussions.}. Hence, by choosing sufficiently small $\e = \e(\lambda)>0$, the Gram matrix $\GG$ also has a negative eigenvalue for some large $N=N(\lambda) \in \mathbb{N}$.

Therefore,  by Definition~\ref{def: pos def} of positive definiteness of kernels, we conclude that the Gaussian kernel $e^{-\lambda d_g^2(\bullet,\bullet)}: \M \times \M \to [0,\infty[$ cannot be positive definite. Here $\lambda>0$ and the Riemannian metric $g$ on $\M$ are arbitrary. The theorem is proved.     \end{proof}

Let us comment on why the above proof fails for simply-connected Riemannian manifolds.

Although it remains true that there is a closed geodesic $\gamma_\star$ on simply-connected $(\M,g)$ --- this is the theorem of Lyusternik--Fet \cite{lf}; see also Klingenberg \cite[Appendix \S~A.1]{k} --- we cannot ensure the shortness of $\gamma_\star$ in any sense.  Synge (\cite{sy}; \emph{cf.} also \cite[\S 9, Corollary~3.10(a)]{dc}) proved the simple-connectivity of closed, orientable, even-dimensional Riemannian manifolds with positive sectional curvature by establishing the non-existence of shortest closed geodesics. Heuristically, this is because any loop is null-homotopic in $\M$, so nothing would stop $\gamma_\star$ from shrinking its length. Thus, we do not have an analogue of ($\spadesuit$) in the proof of Theorem~\ref{thm: main} above, for there always exist ``short-cuts'' which decrease the length of $\gamma_\star$. This renders invalid  eigenvalue stability arguments based on Lemma~\ref{lemma: cmo}.







\section{Remarks}   

In this paper, we have proved  that for each parameter $\lambda>0$, the Gaussian kernel $e^{-\lambda d_g^2(\bullet, \bullet)}$ on any non-simply-connected closed Riemannian manifold $(\M,g)$ fails to be positive definite. Indeed, we establish ---- relying crucially on the computations in \cite{cmo}, especially Lemma~\ref{lemma: cmo} --- the existence of a large number $N=N(\lambda)$ such that for $N$ sample points equidistributed on a nontrivial closed geodesic on $\M$, the associated $N \times N$ Gram matrix $\GG \equiv \GG_N$ for the Gaussian kernel is not positive semi-definite.

Nevertheless, \emph{for any fixed $N$} the Gram matrix $\GG_N$  converges in the $L^\infty$-norm to the identity matrix as $\lambda \to \infty$.  Hence, as the parameter $\lambda>0$ increases, we need to take larger number $N \to \infty$ of sample points to detect the failure of positive definiteness of the Gaussian kernel. See  \cite[\S 5]{cmo}.  We are thus led to propose the following
\begin{cq*}
On complete Riemannian manifolds the Gaussian kernel $e^{-\lambda d_g^2(\bullet,\bullet)}$ is positive definite on coarser scales, while for compact manifolds of dimension $\geq 1$, positive definiteness is lost  on finer scales. Formulate and prove a quantitative description of this phenomenon.
\end{cq*}

It should also be remarked (see \cite[\S 3]{cmo}) that one may construct finite metric spaces $\M^0$ (which are $0$-dimensional manifolds) with complicated $\emptyset \varsubsetneq\Lambda_+\left(\M^0\right)\varsubsetneq \, ]0,\infty[$. Also, Sra \cite{sra} proved that for $X={\bf Sym}^{++}(n)$ with  S-divergence, it holds that $\Lambda_+(X) = \bigcup_{i=1}^{n-2}\left\{\frac{i}{2}\right\} \cup \left[\frac{n-1}{2}, \infty\right[$.

\appendix
\section*{Appendix}
\addcontentsline{toc}{section}{Appendices}
\renewcommand{\thesubsection}{\Alph{subsection}}

\subsection{Two notions of isometric embeddings}\label{sec: appendix}

An isometric embedding in the sense of metric spaces between Riemannian manifolds is defined in Definition~\ref{def: text}. Then, Remark~\ref{remark} pointed out that this is very different from the usual (differential geometric) notion of isometric embedding of Riemannian manifolds.  This appendix is devoted to elaborating on this issue. Recall that

\begin{definition*}
Let $(X,d)$ and $(X',d')$ be metric spaces. A map $\iota: X \to X'$ is said to be an isometric embedding if $\iota$ is one-to-one and $d'\big(\iota(x), \iota(y)\big)=d(x,y)$ for all $x, y \in X$. 
\end{definition*}

\begin{definition*}
Let $(\M,g)$ and $(\M',g')$ be complete Riemannian manifolds. We say that $\iota: (\M,g)\emb(\M',g')$ is a metric space isometric embedding if $\iota$ is an embedding between $\CC^\infty$-manifolds, and  $\iota: \left(\M,d_g\right)\emb\left(\M',d_{g'}\right)$ is a metric space isometric embedding.
\end{definition*}

The above definition is \emph{non-local} in nature --- instead of the pointwise coincidence of length of vectorfields in the differential geometric notion of isometric embedding, here one requires the coincidence of geodesic distance, which is defined via the length of curves. 


In contrast, the differential geometric notion of isometric embeddings between Riemannian manifolds, as in the seminal works by Nash~\cite{n1, n2} (see the comprehensive survey~\cite{hh} by Han--Hong), is as follows --- $\iota: (\M,g) \emb (\M',g')$ is an isometric embedding if $\iota$ is an embedding of $\CC^\infty$-manifolds and $g'\big(\dd\iota(X), \dd\iota(Y)\big)=g(X,Y)$ for any vectorfields $X,Y$ on $\M$. This is a \emph{local} condition: it only requires the inner products induced by Riemannian metrics to coincide at each pair of points $\big(p, \iota(p)\big) \in \M \times \M'$; or equivalently, the length of each pair of $\iota$-related vector $X\big|_p \in T_p \M$ and $\dd\iota(X)\big|_{\iota(p)} \in T_{\iota(p)}\M'$ to be identical, taken with respect to $g$ and $g'$.

It should be emphasised that these two notions of isometric embeddings between Riemannian manifolds are essentially distinctive. Consider the following simple example.

\begin{example*}
The round circle $\sph^1$ is clearly isometric embeddable into the 2-dimensional Euclidean space  in the sense of differential geometry,  via the inclusion $\iota$ of the unit circle in $\R^2$. However, $\iota$ is not an isometric embedding in the sense of metric spaces, since any pair of antipodal points $\{p_+, p_-\}$ on $\sph^1$ have intrinsic distance $\pi$, while their extrinsic distance (\emph{i.e.}, with respect to the Euclidean distance on the ambient space) equals $2$ in $\R^2$.

One way to understand this phenomenon is that in $\R^2$, the diameter segment between $p_\pm$  serves as a ``short-cut'' relative to the geodesic distance on $\sph^1$ taken along the circumference. 
\end{example*}

More generally, by Nash's  theorem~\cite{n2}, any Riemannian manifold $(\M,g)$ can be isometrically embedded into some Euclidean space in the differential geometric sense. However, unless the manifold $(\M,g)$ is itself Euclidean flat, such an isometric embedding would never be a metric space isometric embedding, as proved by Feragen--Lauze--Hauberg \cite{flh}.

Examples of Riemannian manifolds admitting metric space isometric embeddings of the round circle $\varrho\sph^1$ for some scaling factor $\varrho>0$ include  $\sph^n$ for $n \geq 1$, $\R{\mathbb{P}}^n$ for $n > 1$, and the real Grassmannians ${\bf Gr}(k,n)$ for $1 \leq k <n$. This, in fact, is how Theorem~\ref{thm: Sn etc} was proved in \cite{cmo}.

We conclude our clarification on the notions of isometric embeddings by the simple observation below, which sheds further lights to the examples above. 

\begin{th*}
Let $(\M,g)$ and $(\M',g')$ be complete Riemannian manifolds.  
\begin{enumerate}
\item
If a submanifold $(\M,g) \subset (\M',g')$ (\emph{i.e.}, the inclusion map $\iota$) is an isometric embedding in the sense of metric spaces, then $\iota$ is necessarily totally geodesic, hence an isometric embedding in the sense of differential geometry.

\item
Let $\pi: (\M,g) \to (\M,g)$ be a surjective isometry from the metric space $(\M,d_g)$ to itself. Then $\pi$ is isometric in the differential geometric sense.  In particular, if $f: (\M,g) \to (\M,g)$ is a metric space isometric embedding and is surjective, then it is an isometric embedding in the differential geometric sense. 
\end{enumerate}
\end{th*}

\begin{proof}
By Definition~\ref{def: text}, being a metric space isometric embedding, $\iota: (\M,g) \emb (\M',g')$ preserves geodesic distances. In particular, it takes geodesics to geodesics. This proves the first statement. The second statement is the content of a theorem by Myers--Steenrod \cite{ce}.  \end{proof}

\medskip

\noindent
{\bf Acknowledgement}.
The research of SL is supported by NSFC (National Natural Science Foundation of China) Grant No.~12201399, SJTU--UCL joint seed fund, and Shanghai Frontiers Science Center of Modern Analysis.

\end{document}